\documentclass[a4paper,12pt,reqno]{amsart}
\usepackage{amsmath,amssymb,amsthm}
\usepackage{mathrsfs}
\usepackage{graphics}
\usepackage{graphicx,color} 
\usepackage{epsfig}
\theoremstyle{plain}
\newtheorem{thm}{Theorem}[section]
\newtheorem{corollary}[thm]{Corollary}
\newtheorem{lemma}[thm]{Lemma}
\newtheorem{proposition}[thm]{Proposition}

\theoremstyle{definition}
\newtheorem{defn}{Definition}[section]

\theoremstyle{remark}

\numberwithin{equation}{section}
\begin{document}
\begin{center}
\title{Variable Lebesgue algebra on a Locally Compact group }
\maketitle
Parthapratim Saha$^{a}$ and Bipan Hazarika$^{b,}$\footnote{The corresponding author}

\vspace{.2cm}
$^{a}$Department of Mathematics, Sipajhar College, Sipajhar,  Darrang-784145, Assam, India\\
Email:  pps\_gu$@$gauhati.ac.in; parthasaha.sipclg$@$gmail.com

\vspace{.2cm}

$^{b}$Department of Mathematics, Gauhati University, Guwahati-781014, Assam, India\\

 Email:  bh\_rgu$@$yahoo.co.in; bh\_gu$@$gauhati.ac.in
\end{center}
\vspace{.5cm}
\title{}
\author{}

\begin{abstract}
For a locally compact group $H$ with a left Haar measure, we study  variable Lebesgue algebra $\mathcal{L}^{p(\cdot)}(H)$ with respect to a  convolution. We show that if $\mathcal{L}^{p(\cdot)}(H)$ has bounded exponent, then it contains a left approximate identity. We also prove a necessary and sufficient condition for $\mathcal{L}^{p(\cdot)}(H)$ to have an identity. We observe that a closed linear subspace  of $\mathcal{L}^{p(\cdot)}(H)$ is a left ideal if and only if it is left translation invariant.
 \end{abstract} 

\section{\textbf{Introduction}}

The $L^p$-spaces on locally compact group have very essential role in Harmonic analysis. It has numerous applications in Mathematics, Physics, Electrical Engineering and other branches. 
The space $L^p(H)$, where $H$ a locally compact group and $1 \leq p < \infty$,  is a Banach space and even a Banach algebra with respect to a convolution if $H$ is compact  \cite{Saeki}. By replacing the constant exponent $p$ with a variable exponent including minimal restriction, the many results on  classical Lebesgue space holds there. The theory of variable exponent space was initially introduced by Orlicz in 1930's, these theories was further studied and analyzed by various authors. However, many interesting results  were found in variable exponent space and a spike of interest was  found among the authors in the recent years. 

The variable Lebesgue space is a generalization of the  $L^p$-spaces for constant exponent. It was first put forward by the Russian mathematician Tsennov \cite{Tsenov} and further expanded by Sharapudinov \cite{Sharapudinov1, Sharapudinov2, Sharapudinov3, Sharapudinov4} and Zhikov \cite{Zhikov1,Zhikov2,Zhikov3,Zhikov4,Zhikov5,Zhikov6}. In 1991, Kov\'{a}\v{c}ik and J. R\'{a}kosn\'{i}k studied some fundamental properties of the variable Lebesgue space in \cite{OJ}, which is considered as one of the foundational paper in this topic. Orlicz space is a particular case of variable Lebesgue space in some extent. In \cite{Peter}, the author established that the variable Lebesgue space is a Banach space. He also proved  some elementary results of classical Lebesgue space in variable Lebesgue space. Motivated by the above literature, we investigate that Lebesgue space of variable type,  forms a Banach algebra with respect to a suitable convolution. We organize the article  in the following way:

Some preliminary results and definition of variable Lebesgue space are furnished in Section 2. In Section 3, we provide the condition under which the variable Lebesgue space on a locally compact group is a Banach algebra with respect to a suitable convolution, which is mentioned as variable Lebesgue algebra. The left approximate identity variable Lebesgue algebra $\mathcal{L}^{p(\cdot)}(H)$ prove in Section 4. 
 We also prove   a necessary and sufficient condition for the existence of an identity   in the variable Lebesgue algebra.  A familiar result on the weighted Orlicz algebra  $L^{\phi}_w (H)$ (see \cite{AlenSerap}) that the closed left ideal of variable Lebesgue algebra $\mathcal{L}^{p(\cdot)}(H)$ is nothing else the  left translation invariant subspace of $\mathcal{L}^{p(\cdot)}(H)$ is also proved.

\section{\textbf{Preliminaries}}
Here we consider $H$ a locally compact group with a Haar measure $\mu$. We recall the following definitions and results: 
 \begin{defn}
\cite{Chon}  A non zero regular Borel measure measure $\mu$ on a locally compact group $H$ is called a left Haar measure if it is left invariant under left translation i.e., $\mu(Sg) =\mu((S)$ for all $g \in H $ and all Borel subset $S \subset H$.

In the similar way, one can define right Haar measure.
  \end{defn}
The set of all $\mu$-measurable functions from $H$ into $\mathbb{C}$ or into the expended real numbers $[-\infty , \infty]$ is denoted by $\mathcal{M}$. An element of $\mathcal{M}$ is called an exponent if it takes the values in $[1, \infty]$. We denote the  set of all exponents  by $\mathcal{E}$.   For a fix $p \in \mathcal{E}$ and any $f \in \mathcal{L}^{p(\cdot)}(H)$, the functional $ \rho_p$ defined by 
\begin{equation}\label{e1}
\rho_p(f) := \int_{F_p} |f(x)|^{p(x)}d\mu(x) + \operatorname{ess}\sup\limits_{x\in F_p^c} |f(x)|,
\end{equation}
where $F_p= \{ x \in H : p(x) < \infty\}$, is called the $p$-modular function. It ws first introduced by Kov\'{a}\v{c}ik and R\'{a}kosn\'{i}k    \cite{OJ}. There are two other definitions                                                                                                                                                                                                                                                                                                                                                                                                                                                                                                                                                                                                                                                                                                                                                                                                                                                                                                                                                             of the modular to define variable Lebesgue space. One of them is defined as 
\begin{equation}\label{e2}
\rho_p :=\max\left( \int_{F_p} |f(x)|^{p(x)}d\mu(x),~\operatorname{ess}\sup\limits_{x\in F_p^c} |f(x)|\right).
\end{equation}
It can be easily seen that the modular (\ref{e2}) is equivalent to modular given by (\ref{e1}) and they admit same norm. The modular (\ref{e2}) was introduced by Edmunds and R\'{a}kosn\'{i}k   \cite{EdmundsRakosnik}. Another approach to define a modular inspired by the theory of Musielak-Orlicz space  is given in \cite{Diening} as 
\begin{equation}
\rho_p(f)= \int_H |f(x)|^{p(x)}dx
\end{equation} 
with the convention that $t^{\infty}= \infty .\chi_{(1,\infty)}(t)$. This modular is not equivalent to (\ref{e1}), but the resulting norm is equivalent.

For a locally compact group $H$ and $p \in \mathcal{E}$, the variable Lebesgue space $\mathcal{L}^{p(\cdot)}(H)$ is defined as (see \cite{Peter})
\begin{equation*}
\mathcal{L}^{p(\cdot)}(H) = \left\{ f \in \mathcal{M} : \rho_p \left(\frac{f}{\lambda}\right) < \infty ,~\exists \lambda > 0 \right\}. 
\end{equation*}
Then the variable Lebesgue space $\mathcal{L}^{p(\cdot)}(H)$ is a Banach algebra under the norm $\|f\|_{p(\cdot)}$ defined by for all $f \in \mathcal{L}^{p(\cdot)}(H)$ 
\begin{equation*}
\|f\|_{p(\cdot)} = \inf \left\{ t >0 : \rho_p(f / t) \leq 1\right\}.
\end{equation*}
There is another norm  $\|f\|_{p(\cdot)}^A$ (Amemiya norm) on  $\mathcal{L}^{p(\cdot)}(H)$ defined as, for any $p \in \mathcal{E}$ with $p_{+} = \operatorname{ess}\sup\limits_{H}p(x) < \infty$ 
\begin{equation*}
\|f\|_{p(\cdot)}^A = \inf\left\{k > 0 : k + k \rho_p\left(\frac{f}{k}\right)\right\}.
\end{equation*}
The above two norm are equivalent and $\|f\|_{p(\cdot)} \leq \|f\|_{P(\cdot)}^A \leq 2\|f\|_{p(\cdot)}$ (for proof see \cite{S. Samko}).

 An exponent $p \in \mathcal{E}$ is called a bounded exponent if $p_{+} = \operatorname{ess}\sup\limits_H p(x) < \infty $. In \cite{Peter}, Peter showed that  $\rho_p$-convergence and $\|.\|_{p(\cdot)}$-convergence  in $\mathcal{L}^{p(\cdot)}(H)$ are equivalent, if $p$ is a bounded exponent. Denote $\mathcal{S}$  the class of simple functions that vanishes outside of a set of  finite measure, then $\mathcal{S}$ is dense in $\mathcal{L}^{p(\cdot)}$, whenever $p$  is a bounded exponent (see \cite[Theorem 1.5]{Peter}). Given any $p(\cdot) , q(\cdot) \in \mathcal{E}$ and if $\mu(H) < \infty$ then $\mathcal{L}^{q(\cdot)}(H) \subset \mathcal{L}^{p(\cdot)}(H) $ when $p(\cdot) \leq q(\cdot)$ $\mu$-almost everywhere.  Moreover, there exist constants $c_1, c_2 > 0$  such that 
\begin{equation}\label{e3}
 c_1\|f\|_{p_-} \leq \|f\|_{p(\cdot)} \leq c_2 \|f\|_{p_+}, 
\end{equation}  
where $p_{-} = \operatorname{ess}\inf\limits_H p(x)$ and $p_{+ }= \operatorname{ess}\sup\limits_H p(x)$ (see \cite[Theorem 2.26 and Corrollary 2.27]{Uribe}).

The H\"{o}lder's inequality holds in variable Lebesgue space. If $p$ and $q$ are exponent and $r$ is the function define by $\frac{1}{r(x)} = \frac{1}{p(x)} + \frac{1}{q(x)}$ is an exponent, then there exists a constant $K \in [1,5]$ such that $\|fg\|_{r(\cdot)} \leq K \|f\|_{p(\cdot)} \|g\|_{q(\cdot)}$. The conjugate exponent $p'$ of $p$ is the exponent that satisfies the equation $\frac{1}{p(x)} + \frac{1}{q(x)} =1 $ for all $x \in X $. The functional $\|\cdot\|_{p(\cdot)}^{'}$ defined on  $\mathcal{M},$ given by 
\begin{equation*}
\|f\|_{p(\cdot)}^{'} :=  \sup_{\rho_{p'}(g) \leq 1}\int_H |fg|d\mu
\end{equation*} 
is the conjugate norm to $\|\cdot\|_{p(\cdot)}$. Various other definitions used by different author for $\|\cdot\|_{p(\cdot)}^{'}$ is characterized by the following equalities \cite[Proposition 1.8]{Peter}:
\begin{align*}
\|f\|_{p(\cdot)}^{'} &= \sup_{\|g\|_{p'(\cdot)} \leq 1}\int_H |fg|d\mu = \sup_{\rho_{p'}(g) \leq 1} \left|\int_H fgd\mu \right| = \sup_{\|g\|_{p'(\cdot)} \leq 1}\left|\int_H fgd\mu\right|\\
&= \sup_{\rho_{p'}(g) \leq 1, ~g \in \mathcal{S}}\int_H |fg|d\mu = \sup_{\|g\|_{p'(\cdot)} \leq 1, ~g \in \mathcal{S}}\int_H |fg|d\mu ,
\end{align*}  
where $\mathcal{S}$ is the class of simple functions on $H$. If $\mu $ is $\sigma $-finite and $p$ is a bounded exponent, then $\mathcal{L}^{p'(\cdot)}(H)$ is the dual space of $\mathcal{L}^{p(\cdot)}(H)$, where $p'$ is the conjugate exponent to $p$. Moreover, if $p_- = \operatorname{ess}\inf\limits_H p(x) >1$, then $\mathcal{L}^{p(\cdot)}(H)$ is reflexive.

Remark that a left approximate identity in a Banach algebra $\left(A, \|\cdot\|\right)$ is a net $(e_{\alpha})_{\alpha \in \Lambda}$ in $A$ if $ \lim_{\alpha}\|e_{\alpha}x - x\| =0 $ for all $ x \in A$, similarly right approximate identity is also defined.
Now we mention some properties of the norm $\|\cdot\|_{p(\cdot)}$ on $\mathcal{L}^{p(\cdot)}(H)$ which are useful in our studies as follows:
\begin{proposition}\cite[Proposition:1.2]{Peter}
Let $p \in \mathcal{E}$ and $f, g \in \mathcal{L}^{p(\cdot)}$, then
\begin{enumerate}
\item[(i)] $\|f\|_{p(\cdot)} \geq 0 $; the equality holds iff $f=0$, $\mu$-a.e.
\item[(ii)] If $|f| \leq |g|, \mu$-a.e, then $\|f\|_{p(\cdot)} \leq \|g\|_{p(\cdot)}$.
\item[(iii)] $\rho_p\left(f/\|f\|_{p(\cdot)}\right) \leq 1$, provided $f \neq 0$ $ \mu$-a.e.
\item[(iv)] $\rho_p(f)\leq 1$  iff $\|f\|_{p(\cdot)} \leq 1$.
\item[(v)] If $\rho_p(f)\leq 1$ or $\|f\|_{p(\cdot)} \leq 1$, then $\rho_p(f) \leq \|f\|_{p(\cdot)} $.
\item[(vi)] $\rho_p(f)\geq 1$ or $\|f\|_{p(\cdot)} \geq 1$, then $\|f\|_{p(\cdot)} \leq \rho_p(f)$.
\end{enumerate}
\end{proposition}

\section{\textbf{Variable Lebesgue Algebra }}
 		In this section, we discuss that for a locally compact group $H$ with a Haar measure $\mu $, variable Lebesgue space $\mathcal{L}^{p(\cdot)}(H)$ forms  an algebra with respect to a suitable convolution under certain condition. Before introducing the variable Lebesgue algebra we discuss some lemmas.\\
We define a convolution on $\mathcal{L}^{p(x)}(H)$ as follows. For any $f,g \in \mathcal{L}^{p(x)}(H)$ 
\begin{equation}\label{C}
(f\star g) (x)= \int_H f(y)g(y^{-1}x)d\mu(y).
\end{equation}

\begin{lemma}\label{L1}
$\mathcal{L}^{p(\cdot)}(H) \subseteq L^1(H) $ iff $ \exists$ a constant $k > 0$ such that $\parallel f \parallel_{1} \leq  k \parallel f \parallel _{p(\cdot)}.$ 
\end{lemma}
\begin{proof}
Let us assume that $\mathcal{L}^{p(\cdot)}(H) \subseteq L^1(H)$ on $\mathcal{L}^{p(\cdot)}(H)$. It is easy to see that $(\mathcal{L}^{p(\cdot)}(H), \parallel \cdot\parallel)$ is a Banach space with $\| f \|= \|f\|_1 + \|f\|_{(\cdot)}$. Let $I$ be the identity map from 	 $(\mathcal{L}^{p(\cdot)}(H),  \parallel \cdot \parallel )$ to  $(\mathcal{L}^{p(\cdot)}(H),  \parallel \cdot \parallel_{p(\cdot)} ).$ Then $I$ is continuous and one to one map,  so by open mapping theorem $I$ is an open map and so it is a homeomorphism. Hence $I^{-1}$ is bounded map from  $(\mathcal{L}^{p(\cdot)}(H),  \parallel \cdot \parallel_{p(\cdot)} )$ to 	$(\mathcal{L}^{p(\cdot)}(H),  \parallel \cdot \parallel )$. Thus $\exists $ a constant $k \geq 0$ such that $ \| I^{-1}(f)\| \leq  k \|f\|_{p(\cdot)}$ i.e., $ \|f\| \leq k \|f\|_{p(\cdot)}$. But $\| f \|= \|f\|_1 + \|f\|_{(\cdot)}$, so $\|f\|_1 \leq \|f\|$ and hence $\| f\|_1 \leq k \|f\|_{p(\cdot)}$.
\end{proof}

\begin{lemma}
$\|L_x f\|_{p(\cdot)} = \|f\|_{p(\cdot)} $	$\forall f \in \mathcal{L}^{p(\cdot)}(H)$.
Here   $L_xf(y)= f(x^{-1}y)$ for all $x, y \in H$.
\end{lemma}

\begin{thm}
Let $H$ be a locally compact group with a left Haar measure $\mu $. If $\mathcal{L}^{p(\cdot)}(H) \subseteq L^1(H) $ and $ | f | \leq 1 $ for all $ f \in \mathcal{L}^{p(\cdot)}(H) $, then the variable Lebesgue space $\mathcal{L}^{p(\cdot)}(H)$ is a Banach algebra with respect to the convolution define by (\ref{C})
\end{thm}
\begin{proof}
Let $\mathcal{L}^{p(\cdot)}(H) \subseteq L^1(H) $. Then, by Lemma \ref{L1}, $ \exists c > 0 $ such that 
\begin{equation} \label{E1}
\parallel f \parallel_{1} \leq c\parallel f \parallel _{p(\cdot)}
\end{equation}
for all $f \in \mathcal{L}^{p(\cdot)}(H)$. \\ 

Now for any $ f, g \in \mathcal{L}^{p(\cdot)}(H)$ and $t > 0,$ we have 
\begin{align*}
\rho_p\left(\frac{f \star g}{t} \right)&= \rho_p\left(\frac{1}{t} \int_H f(y)g(y^{-1}x)d\mu(y)\right)\\&= \int_{F_p} \left| \frac{1}{t}\int_H f(y)g(y^{-1}x)d\mu(y)\right|^{p(x)}d\mu (x) + \operatorname{ess}\sup\limits_{x\in F_p^c} \left|\frac{1}{t}\int f(y)g(y^{-1}x)d\mu (y)\right|\\
 & \leq \int_H |f(y)|d\mu(y) \left\{ \int_H \left|\frac{L_yg(x)}{t} \right|^{p(x)}d\mu(x) + \operatorname{ess}\sup\limits_{x\in F_p^c}  \left|\frac{L_yg(x)}{t}\right|\right\}\\
 &= \|f\|_1 \rho_p(L_y(g)/t)
\end{align*}
the last inequality follows from the Fibuni's theorem and the condition $|f| \leq 1$ for all $f \in \mathcal{L}^{p(\cdot)}(H)$. Thus we have 
\begin{equation}\label{E2}
\rho_p(f \star g / t ) \leq \|f\|_1 \rho_p(L_y(g)/t).
\end{equation}
From (\ref{E1}) and (\ref{E2}), it implies that $$\{ t > 0 : \|f\|_1 \rho_p(L_y(g)/t) \leq 1 \} \subseteq \{ t > 0 : \rho_p(f \star g / t ) \leq 1 \},$$ \\ i.e., $\|f \star g \|_{p(\cdot)} \leq \|f\|_1 \|L_yg\|_{p(\cdot)}$ and hence $\|f \star g \|_{p(\cdot)} \leq \|f\|_{p(\cdot)} \|g\|_{p(\cdot)}.$
\end{proof}

\section{\textbf{Approximate identity}}
\begin{lemma}
If $p$ is a bounded exponent and $\mu(H) < \infty$, then the set $C_c(
H)$ of continuous functions with compact support is dense in $\mathcal{L}^{p(\cdot)}(H)$.
\begin{proof}
Let  $\mathcal{S} = \{ s: H \rightarrow \mathbb{C}\mid s~ \mbox{is~ simple,~ measurable~ and~} \mu\{x \mid s(x)\neq 0\} \mbox{~is~ finite} \}.$\\
Then $\mathcal{S}$ is dense in $\mathcal{L}^{p(\cdot)}(H)$   \cite[Theorem 1.5]{Peter}. We claim that $C_c(H)$ is dense in $S$. If $s \in S$ and $\epsilon > 0 $, then by Lusin's theorem $ \exists  g\in C_c(H)$ such that $\mu\{x \mid g(x) \neq s(x)\} < \epsilon^{p^+} $, where $p^{+} = \operatorname{ess}\sup\limits_{x\in H } p(x)$    and $|g| \leq \|s\|_{\infty}$. Since $p$ is bounded exponent, so $p^{+} = \operatorname{ess}\sup\limits_{x\in H } p(x) < \infty $, then by  (\ref{e3}), we have
\begin{align*}
\|g-s\|_{p(\cdot)} \leq c \|g-s\|_{p^+} &= c ~\left(\int_G |g-s|^{p^+}d\mu\right)^{1/p^{+}} \\
&= c ~ \left(\int_{\{x: g(x)\neq s(x)\}}|g-s|^{p^+}d\mu\right)^{1/p^+}\\
& \leq 2c\|s\|_{\infty}~ \epsilon .
\end{align*}
This implies that $C_c(H)$ is dense in $S$ and since $S$ is dense in $\mathcal{L}^{p(\cdot)}(H)$, hence $C_c (H)$ is dense in $\mathcal{L}^{p(\cdot)}(H)$.
\end{proof}
\end{lemma}
\begin{lemma} \label{L2}
Let $ p $ be a bounded exponent and $\mathcal{L}^{p(\cdot)}(H)$ is a Banach algebra. Then for any $ f \in \mathcal{L}^{p(\cdot)}(H)$ and $\epsilon > 0 $, there exists a neighbourhood $U$ of the identity  such that for all $x\in U$, $\|L_xf -f\|_{p(\cdot)} < \epsilon $.
\end{lemma}
\begin{proof}
Consider a compact neighbourhood $\mathcal{K}$ of the identity of $H$. Let $f \in C_c(H)$ with compact support $S$, then $ \operatorname{supp}(L_xf)=xS,~\forall x \in \mathcal{K}$. Let $\mathcal{K'} = \mathcal{K} \cup S \cup \mathcal{K}S $. There exists a constant $c> 0$ such that 
\begin{align*}
\|L_xf - f\|_{p(\cdot)} &\leq c \|L_xf -f\|_{p^{+}}\\
&= c \left(\int_{\mathcal{K'}}|L_x f - f|^{p^+}d\mu\right)^{1/p^+}\\
& \leq c \|L_xf -f\|_\infty ~ (\mu (\mathcal{K'}))^{1/p^+} .
\end{align*}
Since $\operatorname{supp}(L_xf-f) \subseteq \mathcal{K'}$, $\|L_xf-f\|_\infty < \frac{\epsilon}{2c(\mu(K))^{1/p+}}$ for any $x$ in a sufficiently small neighbourhood $V \subseteq \mathcal{K}$ of identity. Thus
\begin{equation}\label{E3}
 \|L_xf - f\|_{p(\cdot)} < \frac{\epsilon}{2}
\end{equation}
Now, let $f \in \mathcal{L}^{p(\cdot)}(H)$ be arbitrary and  $U$ be compact neighbourhood of identity. Since $C_c(H)$ is dense in $\mathcal{L}^{p(\cdot)}(H)$ for bounded exponent $p$, for $\epsilon >0$,  $\exists$ a $g \in C_c(H)$ such that $\|f-g\|_{p(\cdot)} < \frac{\epsilon}{4}$. Also, as $g \in C_c(H)$ by (\ref{E3}), we get $\|L_xg-g\|_{p(\cdot)} < \frac{\epsilon}{2}$ for all $x \in U$.\\
Thus for any $x \in U$
\begin{align*}
\|L_xf-f\|_{p(\cdot)} &\leq \|Lxf-L_xg\|_{p(\cdot)}+\|L_xg-g\|_{p(\cdot)}+\|g-f\|_{p(\cdot)}\\
&= 2\|f-g\|_{p(\cdot)}+\|L_xg-g\|_{p(\cdot)}\\
& < 2. \frac{\epsilon}{4}+ \frac{\epsilon}{2} = \epsilon .
\end{align*}
\end{proof}
\begin{thm}\label{T1}
Let $\mu(H) < \infty $ and $\mathcal{L}^{p(\cdot)}(H)$ be a variable Lebesgue algebra. Then $\mathcal{L}^{p(\cdot)}(H)$ has a left approximate identity.
\end{thm}
\begin{proof}
Consider a compact neighbourhood $\mathcal{K}$  of the identity element $e \in H $ and $\Im $ be the family of all neighbourhood $U \subseteq K$ of  the identity. Then $\Im $ is a directed set with the pre-order inclusion. Let  $\xi_U = \frac{\chi_U}{\mu(U)}$ for all $U \in \Im $. Then $ \exists  c>0$ such that
\begin{align*}
\|\xi_U\|_{p(\cdot)} &\leq c \|\xi_U\|_{p^+}\\
&= c \|\chi_U/\mu(U)\|^{p^+} \\
&= c \left(\int_G |\frac{\chi_U}{\mu(U)}|^{p+}\right)^{1/p^+}\\
&= c \left(\int_U (1/\mu(U))^{p^+}d\mu\right)^{1/p^+}\\
&= c \left(\mu(U)\right)^{\frac{1-p^+}{p^+}}\\
\Rightarrow \|\xi_U\|_{p(\cdot)} & < \infty .
\end{align*} 
Thus $\rho_p\left(\xi_U\right) \leq \|\xi_U\|_{p(\cdot)} < \infty $ i.e., $\xi_U \in L^{p(\cdot)}(G)$ and $(\xi_U)_{U \in \Im}$ is a net in $L^{p(\cdot)}(G)$.\\
Now, given $f \in L^{p(\cdot)}(G)$ and $\epsilon >0$, by the Lemma $\ref{L2}$, we get  a neighbourhood $U \in \Im$ such that $\|L_tf-f\|_{p(\cdot)} < \epsilon $ for all $t \in U$. Then for all $V \geq U$ with $V \in \Im$, we get have using Fibuni's theorem that 
\begin{align*}
\rho_p\left(\frac{\chi_V \ast f -f}{\epsilon}\right) &= \frac{1}{\epsilon}\left(\int_{F_p}\left|\int_V \frac{f(t^{-1}x)-f(x)}{\mu(V)}d\mu(t)\right|^{p(x)}d\mu(x)\right.\\& \left.+ \operatorname{ess}\sup\limits_{x \in F_p^c}\left|\int_V \frac{f(t^{-1}x)-f(x)}{\mu(V)}d\mu(t)\right|\right)\\
& \leq \frac{1}{\epsilon}\frac{1}{\mu(V)}\left\{\int_{F_p}\int_V \left|f(t^{-1}x)-f(x)\right|^{p(x)} d\mu(t)d\mu(x)\right.\\&
 \left.+  \operatorname{ess}\sup\limits_{x\in F_p^c} \left(\int_V \left|f(t^{-1}x)-f(x)\right|d\mu(t)\right)  \right\}\\
&= \frac{1}{\epsilon}\frac{1}{\mu(V)}\left\{\int_V\left(\int_{F_p}\left|L_tf-f\right|^{p(x)}d\mu(x)+\operatorname{ess}\sup\limits_{x\in F_p^c}\left|L_tf-f\right|\right)d\mu(t)\right\}\\
&= \frac{1}{\epsilon}\frac{1}{\mu(V)}\int_V \rho_p\left(L_tf-f\right)d\mu(t)\\
&= \frac{1}{\epsilon}\rho_p\left(L_tf-f\right).
\end{align*} 
Since $\rho_p\left(L_tf-f\right)\leq \|L_tf-f\|_{p(\cdot)} < \epsilon$, so we get $\rho_p\left(\frac{\xi_V \ast f-f}{\epsilon}\right)\leq 1$. Thus by the definition of the norm $\|\cdot\|_{p(\cdot)}$, we get $\|\xi_V \ast f -f\|_{p(\cdot)} < \epsilon $. Hence $\{\xi_V\}_{V\in \Im}$ is a left approximate identity in $\mathcal{L}^{p(\cdot)}(H)$.
\end{proof}
 In the similar approach we can find right approximate identity in $\mathcal{L}^{p(\cdot)}(H)$.\\
Now we shall discuss about the condition under which the Banach algebra  $\mathcal{L}^{p(\cdot)}(H)$ has an identity.
\begin{thm}
For a bounded exponent $p$, the variable Lebesgue algebra $\mathcal{L}^{p(\cdot)}(H)$ contains an identity iff $H$ is discrete. 
\end{thm}
\begin{proof}
Let $\mathcal{L}^{p(\cdot)}(H)$ be a variable Lebesgue algebra which contains an identity $h$. Then $f \ast h = h \ast f = f$ for all $f \in \mathcal{L}^{p(\cdot)}(H)$. Suppose a neighbourhood $U$ of the identity $ e \in H $ and $\epsilon > 0$. Since $C_c (H)$ is dense in $\mathcal{L}^{p(\cdot)}(H)$, so $\exists f \in C_c(H)$ such that $\operatorname{supp}(f) \subset U$ and $\|f \ast h - h \|_{p(\cdot)} < \epsilon $. Since $f \ast h =f$, we have $\|f-h\|_{p(\cdot)} < \epsilon$. \\  Now
\begin{align*}
\epsilon > \|f-h\|_{p(\cdot)} &\leq \rho_p(f-h)\\
&= \int_U \left|f-h \right|^{p(x)}d\mu + \int_ {H\setminus U}|f-h|^{p(x)}d\mu\\
&\geq \int_{H \setminus U}|h|^{p(x)}d\mu.
\end{align*}   
This implies that $|h(x)|^{p(x)}$ should be zero for all $x\in H \setminus U$. But $1 \leq p(x)< \infty,$  so $g(x)=0$ for all $x \in H \setminus U$. For any neighbourhood $U$ of the identity in $H$, we have $\operatorname{supp}(g)\subset \{e\}$ and $\mu (\{e\})> 0$, as if $\mu (\{e\})= 0$, then $h=0$ a.e. on $H$, which is a contradiction to fact that $h$ ia an the identity in $\mathcal{L}^{p(\cdot)}(H)$. Thus $H$ is discrete.

 Conversely, let $H$ be a  discrete group. Then $\mu$ is a counting measure on $H$. The characteristic function $\chi_e$ of  $\{e\}$ belongs to $\mathcal{L}^{p(\cdot)}(H)$ and we have
\begin{align*}
(\chi_e \ast f)(x)&= \int_H \chi_e(y)f(y^{-1}x)d\mu(y)\\
&= \sum_{x\in H}\chi_e(y)f(y^{-1}x)\\
&=f(x)
\end{align*}
 for all $f\in \mathcal{L}^{p(\cdot)}(H)$ and $x \in H $. Thus the function $\chi_e$ is an identity of the algebra.
\end{proof}
 Since $\mathcal{L}^{p(\cdot)}(H)$ be variable Lebesgue algebra, where $p$ is a bounded exponent. Then using the existence of of a left approximate identity in $\mathcal{L}^{p(\cdot)}(H)$, we perceive that the closed left ideal of $\mathcal{L}^{p(\cdot)}(H)$ is left translation invariant.
\begin{thm}
Suppose $p$ is a bounded exponent and $\mu$ is $\sigma$-finite. Let $\mathcal{L}^{p(\cdot)}(H)$ be a variable Lebesgue algebra and $M$ be a closed linear subspace of $\mathcal{L}^{p(\cdot)}(H)$. Then $M$ is a left ideal in $\mathcal{L}^{p(\cdot)}(H)$ iff $L_x(M) \subseteq M $ for all $x \in H$.
\end{thm}
\begin{proof}
 Let $M \subseteq  \mathcal{L}^{p(\cdot)}(H)$ be a left ideal. For any $f \in M$ and $\epsilon >0$, using Theorem \ref{T1}, we get $\left( e_V\right)_{V \in \Im}$ is a left approximate identity in $\mathcal{L}^{p(\cdot)}(H)$, so that $\|e_V \ast f -f\|_{p(\cdot)} < \epsilon $. Moreover $\left(L_xe_V\right) \ast f \in M$, since $\mathcal{L}^{p(\cdot)}(H)$ is left translation invariant and $M$ is a left ideal. Thus $\|L_xe_V \ast f - L_xf \|_{p(\cdot)} = \|e_V \ast f - f\|_{p(\cdot)} < \epsilon $. Therefore $L_xf \in M$, since $M$ is closed.

Conversely, suppose that $M$ is a left translation invariant subspace of     $\mathcal{L}^{p(\cdot)}(H),$ i.e., $L_x(M) \subseteq M$ for all $x \in H$. To prove that $M$ is a left ideal. To prove this we need to show that $j \ast i \in M $ for all $ i \in M $ and $j\in \mathcal{L}^{p(\cdot)}(H) $. Suppose that there exists $i \in M$ and $j \in \mathcal{L}^{p(\cdot)}(H)
$ such that $j \ast i \notin  M $. Then by the consequence of Hahn-Banach theorem, $\exists$ a bounded linear  functional $\Psi$ on $\mathcal{L}^{p(\cdot)}(H))$ such that $\Psi(M) = \{0\}$ and $F(j \ast i)\neq 0$. Further, since $p$ is a bounded exponent and $\mu$ is $\sigma$-finite, the dual space of $\mathcal{L}^{p(\cdot)}(H)$ is $\mathcal{L}^{p'(\cdot)}(H)$, where $p'$ is the conjugate exponent of $p$. So, the  bounded linear functional $ \Psi \in \left(\mathcal{L}^{p(\cdot)}(H)\right)^* $ can be uniquely determine by $\varphi \in \mathcal{L}^{p'(\cdot)}(H) $, such as $$\Psi(\xi) = \int_H \xi\varphi  d\mu, ~\xi \in \mathcal{L}^{p(\cdot)}(H). $$
Therefore,
\begin{align*}
  \Psi(j \ast i) &= \int_H \varphi(x)(j \ast i)(x)d\mu(x) \\
& = \int_H \varphi(x)\left(\int_H j(y)i(y^{-1}x)d\mu(y)\right)d\mu(x)\\
&= \int_H j(y)\left(\int_H\varphi(x)L_yi(x)d\mu(x)\right)d\mu(y)\\
&= \int_H j(y)\Psi(L_yi)d\mu(y)\\
&=0.
\end{align*}
Since $L_yf \in M $ and $\Psi(M)=\{0\}$, which contradicts our assumption that $\Psi(j \ast i)\neq 0$. This completes the proof.
\end{proof} 

\begin{corollary}
If $M$ is a  subspace of $\mathcal{L}^{p(\cdot)}(H)$ which is closed, then $M$ a right ideal in $\mathcal{L}^{p(\cdot)}(H)$ iff $M \subseteq \mathcal{L}^{p(\cdot)}(H) $ is  right translation invariant.
\end{corollary}


\end{document}